\documentclass[11pt]{amsart}   

\usepackage{verbatim}
\usepackage{amssymb,amsthm,amsmath}
\usepackage{color}
\definecolor{darkblue}{rgb}{0, 0, .4}
\definecolor{grey}{rgb}{.7, .7, .7}

\usepackage{ifpdf}
\ifpdf
  \usepackage[pdftex]{epsfig}

  \usepackage{lscape}

  \usepackage[breaklinks]{hyperref}
  \hypersetup{
            colorlinks=true,
            linkcolor=darkblue,
            anchorcolor=darkblue,
            citecolor=darkblue,
            pagecolor=darkblue,
            urlcolor=darkblue,
            pdftitle={},
            pdfauthor={}
  }
\else
  \usepackage[dvips]{epsfig}
  \usepackage{lscape}

  \newcommand{\href}[2]{#2}
  \newcommand{\url}[2]{#2}
\fi

\newtheorem{theorem}{Theorem}[section]
\newtheorem{lemma}[theorem]{Lemma}

\theoremstyle{definition}
\newtheorem{definition}[theorem]{Definition}
\newtheorem{example}[theorem]{Example}
\newtheorem{warning}[theorem]{Warning}

\theoremstyle{remark}

\numberwithin{equation}{section}

\theoremstyle{theorem}
\newtheorem{corollary}[theorem]{Corollary}

\newtheorem{proposition}[theorem]{Proposition}

\newtheorem{question}[theorem]{Question}

\usepackage{graphicx}
\input xy 
\xyoption{all} 

\newcommand{\lf}{\lfloor}
\newcommand{\rf}{\rfloor}

\newcommand{\abacus}[1]{{ \tiny \xymatrix @-2.2pc { #1 } }}

\newcommand{\ci}[1]{{\xy*{ #1 }*\cir<10pt>{}\endxy}}  
\newcommand{\nc}[1]{{\xy*{ #1 }*i\cir<10pt>{}\endxy}}  

\usepackage{tikz}
\usepackage{ifthen}

\newcommand\opdot[1]{\mathring{#1}}
\newcommand\dunion[0]{\mathop{\dot{\bigcup}}}

\begin{document}

\title{Rational generating series for affine permutation pattern avoidance}

\begin{abstract}
We consider the set of affine permutations that avoid a fixed permutation
pattern.  Crites has given a simple characterization for when this set is
infinite.  We find the generating series for this set using the Coxeter length
statistic and prove that it can always be represented as a
rational function.  We also give a characterization of the patterns for which
the coefficients of the generating series are periodic.  The proofs exploit a
new polyhedral encoding for the affine symmetric group.
\end{abstract}

\author{Brant Jones}
\address{Department of Mathematics and Statistics, MSC 1911, James Madison University, Harrisonburg, VA 22807}
\email{\href{mailto:jones3bc@jmu.edu}{\texttt{jones3bc@jmu.edu}}}
\urladdr{\url{http://educ.jmu.edu/\~jones3bc/}}


\date{\today}

\maketitle


\bigskip
\section{Introduction}\label{s:intro}

The affine symmetric group $\widetilde{S}_n$ is an infinite group that arises
naturally in various geometric, combinatorial, and algebraic contexts.  In this
work, we are concerned with the enumeration of various subsets of this group.
Since the group is infinite, we consider ``refined'' counts of elements based on the
Coxeter length statistic $\ell(w)$ that describes the minimal number of
generators needed to factor $w \in \widetilde{S}_n$ in a certain standard group
presentation of $\widetilde{S}_n$.  We will use the language of {\bf generating
series} to describe our results:  For a given subset $S \subseteq \widetilde{S}_n$, we
form the series $\sum_{w \in S} x^{\ell(w)}$ using a formal variable $x$ and
attempt to find a closed form for this expression.  The associated {\bf enumerating
sequence} is the sequence of coefficients which counts the number of elements of
each given length.  These are related; for example, the enumerating sequence is
given by a linear constant-coefficient recurrence precisely when the generating
series can be expressed as a rational function.

One of the first results in this direction is due to Bott \cite{bott} who
gave a general method to compute the Poincar\'e series that describes the Betti
numbers for the associated compact Lie group.  Combinatorially, this is the
generating series by Coxeter length for the entire group $S = \widetilde{S}_n$.

\begin{theorem}{\cite{bott}}\label{t:bott}
We have
\[ \sum_{w \in \widetilde{S}_n} x^{\ell(w)} = \frac{(1+x)(1+x+x^2)(1+x+x^2+x^3) \cdots (1+x+x^2+x^3+ \cdots +x^{n-1})}{(1-x)(1-x^2)(1-x^3) \cdots (1-x^{n-1})}
\]
\end{theorem}

Although his motivation and proof were topological, it is relatively
straightforward to give a combinatorial proof by induction using the (so-called
``parabolic'') subgroups obtained from subsets of the standard generators (see
\cite[(5.12)]{humphreys}).  These subgroups turn out to be finite symmetric
groups, for which the generating series is given by the numerator in
Theorem~\ref{t:bott}.  We give a new combinatorial proof for the denominator in
Bott's formula in Corollary~\ref{c:bottdenom}.

Recently, Crites gave a natural extension of permutation pattern
avoidance for the affine symmetric group as part of his thesis work with Sara
Billey \cite{billey-crites} to characterize the rationally smooth Schubert
varieties of affine type $A$.  In \cite{crites}, he also enumerated the number
of affine permutations avoiding various fixed patterns, and proved the
following remarkable structure theorem.

\begin{theorem}{\cite{crites}}
Let $p$ be a finite permutation and $n \geq 2$.  There exist only finitely many
affine permutations of size $n$ that avoid $p$ if and only if $p$ avoids the
classical permutation pattern $[321]$.
\end{theorem}

Even when there are infinitely many affine permutations of size $n$
that avoid a fixed pattern $p$, we can still consider the length generating
series 
\[ F_{p, n}(x) := \sum_{\substack{w \in \widetilde{S}_n \\ w \text{ avoids } p}} x^{\ell(w)}. \]
Such series first appeared in Hanusa and Jones' \cite{HJ}
enumeration of the $[321]$-avoiding affine permutations.  It is shown there that
the coefficients of the length generating series for $p = [321]$ are periodic.
These $[321]$-avoiding affine permutations are also known as the fully
commutative elements of affine type $A$.  More recently, Biagioli, Jouhet, and Nadeau
\cite{FPSAC13} have described the length generating series for fully
commutative elements in other affine types, and they turn out to be periodic
there as well.  In fact, they propose the problem of determining which Coxeter
groups have a periodic generating series associated to their subset of fully commutative
elements.  This would generalize Stembridge's classification \cite{s1}.

In this work, we consider the dual problem of classifying the periodic patterns within
the affine symmetric group.  While any generating series with periodic
coefficients can be expressed as a reduced rational function with denominator $1-x^d$,
it is not obvious that the $F_{p,n}(x)$ series are even rational in general.  
One standard way to show that a counting problem is solved by a
rational generating series is to produce a bijection to directed paths in a
finite graph (or equivalently, words in a regular language).  Stanley \cite{ec1}
refers to this as the ``transfer matrix method.''  In fact, Brink and Howlett
have described a clever finite state automaton that recognizes a canonical reduced
expression for each element of a fixed Coxeter group (see \cite{brink-howlett}
or \cite[Chapter 4]{b-b}); Casselman has also contributed significantly to make
their ideas practical for efficient implementation in software (see
\cite{casselman}, for example).  
We initially attempted to modify these constructions to filter the affine permutations
based on pattern avoidance criteria.  At this stage, however, it appears that pattern
avoidance is not sufficiently related to the group structure for this approach to work in general.

Recently, we have turned instead to a set of ideas based on geometric
convexity.  Consider a rational polyhedron $P$ defined as the set of solutions
in $\mathbb{R}^n$ to a set of linear inequalities with integral coefficients,
and suppose that we would like to count the lattice points in $\mathbb{Z}^n \cap
P$.  To be more general, we consider the {\bf encoding series} 
\[ \mathcal{F}_{P}(x_1, \ldots, x_n) := \sum_{(z_1, \ldots, z_n) \in \mathbb{Z}^n \cap P} x_1^{z_1} x_2^{z_2} \cdots x_n^{z_n} \]
for these points in the formal variables $x_1, \ldots, x_n$.  Brion's formula
(see \cite{beck-robins} or \cite{barvinok}) states that this encoding series
is simply the sum of the encoding series for each of the ``tangent cones''
formed by the rays emanating from a vertex of $P$.  Moreover, it is
straightforward to see (after using inclusion-exclusion if the cones are not
simple) that the encoding series for these tangent cones are all rational, and so any
generating series obtained by specializing the $x_i$ will be rational also.

More precisely, we show in Section~\ref{s:polyhedral} how to coordinatize (the minimal length coset
representatives of) $\widetilde{S}_n$ as the set of lattice points $(z_1,
\ldots, z_{n-1})$ in the nonnegative orthant $\mathbb{Z}_{\geq 0}^{n-1}$ with
Coxeter length given by $\sum_{i=1}^{n-1} (n-i) z_i$.  Enumerating these points
recovers the denominator of Bott's formula.

However, it turns out that the subset of lattice points corresponding to
the $p$-avoiding affine permutations, for a fixed pattern $p$, is not necessarily
convex; see Figure~\ref{f:n3ex}(b).  We then show that it is possible to decompose
$\mathbb{Z}_{\geq 0}^{n-1}$ into a disjoint union of $(n-1)!$ shifted, dilated
cones, each of the form
\[ \mathcal{C}_{b}^n := \{ ( t_1, 2 t_2, \ldots, (n-1) t_{n-1}) + ({b}_1, \ldots, {b}_{n-1}) : t_i \in \mathbb{Z}_{\geq 0} \}. \]
If we restrict to each $\mathcal{C}_b^n$, then we can prove that the
$p$-avoiding affine permutations do form a polyhedral set.  In fact, we give
explicit defining inequalities that include some additional
coordinates for convenience, and then project to the $t$-coordinates that parameterize each
$\mathcal{C}_b^n$.  At the end of this process, we can apply Brion's formula to
compute the enumerating series and conclude that it is rational.

Let us pause to mention that this construction seems likely to be useful in
other contexts.  For example, the Coxeter hyperplane arrangement of affine type
$A_{n-1}$ in $\mathbb{R}^n$ is given by $x_i - x_j = k$ for $1 \leq i < j \leq
n$ and $k \in \mathbb{Z}$.  The complement of these hyperplanes in
$\mathbb{R}^n$ is a collection of regions.  It turns out that these regions are
in bijection with affine permutations, and so enumerating these regions using a
statistic defined by counting the number of hyperplanes that separate a region
from a fixed region at the origin results in the same generating series as
Bott's formula.  There is some recent interest
\cite{armstrong-dh,fishel-vazirani} in statistics and generating series for
regions of the extended Shi arrangements (which are a subarrangements of this
one), and affine pattern avoidance may be a useful tool for refining this
geometric picture.

The generating series we have been considering also arise in certain lattice
path enumeration problems; see \cite{barcucci,barcucci_q_bessel,FPSAC13}.  In
fact, the enumeration for $p = [321]$ in \cite{HJ} used a recursive technique
of Bousquet-M\'elou \cite{MBMcolumnconvex} developed for this context involving
$q$-Bessel functions that, while powerful, leaves the generating series in a
form that is somewhat opaque.  Our decomposition of the coordinate space for
these objects into shifted dilated cones seems likely to offer some new
insights into these types of recursive systems.

Once we know that our $F_{p,n}(x)$ generating series are rational, there are
three possibilities for the sequence of coefficients:  they must be eventually
zero, eventually repeat, or are unbounded.  The first case is characterized by
Crites' theorem, and in Section~\ref{s:periodic}, we begin to characterize the
periodic patterns.  We are aided by the fact that it suffices to characterize
the periodic patterns in a single $\mathcal{C}_b^n$ space, with $n = 3$.  Stated in
terms of classical permutation patterns, our result essentially requires $p$ to
avoid an infinite family of patterns from $S_7, S_8, S_{10}, S_{12},
S_{14}, \ldots$; see Figure~\ref{f:generalform} and Theorem~\ref{t:ptc}.

When $p$ cannot be embedded into any element of $\widetilde{S}_n$ then the
generating series $F_{p,n}(x)$ is simply given by Bott's formula, which is not
periodic (unless $n = 2$).  It remains an open problem to give a
characterization in terms of $p$ for when this occurs.  We do not address this
here although there are standard techniques from convex geometry that can be
applied to the polyhedra we define for any particular pattern of interest.

There are many open directions in this area, for both undergraduate and
professional researchers.  Almost any of the classical problems associated with
permutation patterns, such as classification of Wilf equivalence classes,
pattern packing, or asymptotic behavior, could be posed in the affine setting;
see \cite{bona-book} for an introduction to these classical results.  It would
also be interesting to extend our geometric framework to study bivariate
generating series of the form $\sum_{w \in S \subseteq \widetilde{S}_n}
x^{\ell(w)} y^n$.  Moreover, modifying the geometric framework to handle
multiple patterns would allow us to study the $\{[3412],[4231]\}$ class from
\cite{billey-crites} in detail.

\bigskip
\section{Polyhedral Structure}\label{s:polyhedral}

\bigskip
\subsection{A polyhedral encoding of the affine symmetric group}

An {\bf affine permutation of size $n$} is a bijection $w: \mathbb{Z}
\rightarrow \mathbb{Z}$ satisfying $w(i+n) = w(i) + n$ for all $i \in
\mathbb{Z}$, and $w(1) + w(2) + \cdots + w(n) = 1 + 2 + \cdots + n$.  We refer
to the (infinite) image sequence $(\ldots, w(-2), w(-1), w(0), w(1), \ldots)$
of $w$ as its {\bf $\mathbb{Z}$-notation}.  By the first property, we can
completely specify an affine permutation by its {\bf base-window} $[w(1), w(2),
\ldots, w(n)]$.  When we do this, the $\mathbb{Z}$-notation is obtained by
decomposing the image into {\bf windows} of size $n$, where the $i$th window
contains the entries of the base window with each value in the window shifted
by $in$.  (In this paper, we denote window boundaries with a $|$ symbol.)

The {\bf affine symmetric group}
$\widetilde{S}_n$ consists of all the affine permutations of size $n$, with
composition of functions as the group operation.  It follows directly from the
definitions that $[w_1, w_2, \ldots, w_n]$ is the base-window notation for an
affine permutation if and only if $\sum_{i=1}^n w_i = { {n+1} \choose 2}$ and
the residues $(w_i \mod n)$ are all distinct.

As a group, $\widetilde{S}_n$ is generated by the $n$ adjacent transpositions
of entries in the $\mathbb{Z}$-notation (where each transposition acts on all
windows simultaneously).  The minimal number of such transpositions into which
$w$ can be factored is known as the {\bf Coxeter length} of $w$, denoted
$\ell(w)$.

Given a permutation $p \in S_k$ and an affine permutation $w \in
\widetilde{S}_n$, we say that {\bf $w$ contains the pattern $p$} if there exist
positions $i_1 < i_2 < \cdots < i_k$ whose $\mathbb{Z}$-notation values $w(i_1),
w(i_2), \ldots, w(i_k)$ are in the same relative order as $p_1, p_2, \ldots,
p_k$.  Note that these positions need not be restricted to the base window.

When the entries in the base-window notation for $w$ are sorted increasingly,
we call $w$ a {\bf minimal length coset representative}.  (See
\cite{b-b,humphreys} for motivation and details.)  We will denote the subset of
minimal length coset representatives by $\widetilde{S}_n^\circ \subseteq
\widetilde{S}_n$.  Then each $w \in \widetilde{S}_n^\circ$ corresponds to an
{\bf abacus diagram} as follows.  Begin with an array having $n$ columns and
countably many rows.  Label the entry in the $i$th row and $j$th column of the
array by the integer $j + ni$, where $1 \leq j \leq n$.  In figures, we will
draw the rows increasingly up the page, and columns increasingly from left to
right.  Then these labels linearly order the entries of the array, which we
refer to as {\bf reading order}.  We call the entries $\{1+kn, 2+kn, \ldots,
n+kn\}$ the $k$th {\bf level} of the array.  To create our diagram, we
highlight certain entries in the array; such entries are called {\bf beads} and
will be circled in figures.  Entries that are not beads will be called {\bf
gaps}.  To encode $w$, we let the entries in the array corresponding to the
base-window notation for $w$ be beads, and we refer to these as the {\bf
defining beads}.  To complete the diagram, we create beads at all of the
entries below each defining bead, lying in the same column.  All of the other
entries in the diagram are gaps.  We call this completed diagram the {\bf
abacus diagram for $w$}.

\begin{figure}[ht]
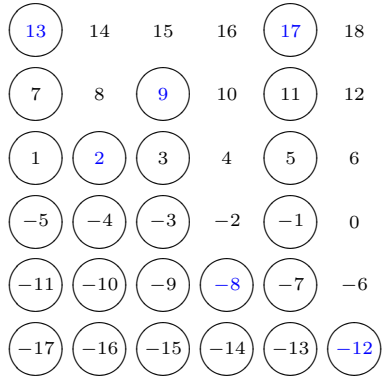
\label{ex:1}
\[
\abacus 
{   
\ci{\color{blue} 13} &  \nc{14} & \nc{15} & \nc{16} & \ci{\color{blue} 17} & \nc{18} \\
\ci{7} &  \nc{8} & \ci{\color{blue} 9} & \nc{10} & \ci{11} & \nc{12} \\
\ci{1} &  \ci{\color{blue} 2} & \ci{3} & \nc{4} & \ci{5} & \nc{6} \\
\ci{-5} & \ci{-4} & \ci{-3} & \nc{-2} & \ci{-1} & \nc{0} \\
\ci{-11} &  \ci{-10} & \ci{-9} & \ci{\color{blue} -8} & \ci{-7} & \nc{-6} \\
\ci{-17} &  \ci{-16} & \ci{-15} & \ci{-14} & \ci{-13} & \ci{\color{blue} -12} \\
}
\]
\caption{An abacus diagram for $w = [-12, -8, 2, 9, 13, 17]$ with $\bar{w} = (4, 10, 7, 4, 4)$ and $\opdot{w} = (0, 3, 3, 2, 3)$.}
\end{figure}

Observe that the defining conditions on the base-window notation imply that the
levels of the defining beads in an abacus diagram must sum to zero.  We refer
to this by saying that the abacus diagram must be {\bf balanced}.  Hence, the
base window notation includes a redundant coordinate.  To remedy this, we can
represent any minimal length coset representative $w$ by its {\bf gap vector}
$\opdot{w} = (\opdot{w}_1, \ldots, \opdot{w}_{n-1})$ where $\opdot{w}_i$
records the number of gaps between the $i$ and $(i+1)$st defining beads in the
abacus for $w$, ordered increasingly.  Alternatively, we may specify $w$ by its
{\bf delta vector} 
\[ \bar{w} = (w_2 - w_1, w_3 - w_2, \ldots, w_n - w_{n-1}). \]
This vector records the number of entries (which may be beads or gaps) in the
abacus diagram between each successive pair of defining beads.

Observe that {\em any} nonnegative integer vector is the gap vector for a
unique abacus diagram.  To see this, simply place the largest defining bead
arbitrarily on the array, and then place each of the smaller defining beads
with consecutive distances as prescribed by the given gap vector.  To balance
the abacus, subtract the sum of the levels of the defining beads from the
position of each defining bead.  The result will be the unique balanced abacus
having the prescribed gap distances between consecutive defining beads.

\begin{proposition}\label{p:coxeter_length}
    The Coxeter length of $w$ is given by $\ell(w) = \opdot{w}_{n-1} + 2\opdot{w}_{n-2} + 3\opdot{w}_{n-3} + \cdots + (n-1)\opdot{w}_{1}$.
\end{proposition}
\begin{proof}
This is a ``folklore'' result that is sometimes stated in a slightly
different form:  To compute the Coxeter length of the element encoded by
an abacus diagram, count the number of pairs $(b, g)$ where $b$ is a
defining bead and $g$ is a gap that preceeds $b$ in reading order.  

Once we translate the action of $\widetilde{S}_{n}$ to the abacus, it is
straightforward to prove this result by induction on $\ell(w)$; simply check
that each length increasing adjacent transposition adds a single new $(b,g)$
pair.
\end{proof}

As a corollary to this development, we may view the (gap vectors of) minimal length coset representatives as lattice points in the nonnegative orthant, which is a prototype for our polyhedral encoding.  
When we enumerate these points with respect to the Coxeter length statistic, we recover the classical result of Bott for type $\widetilde{A}$.  This seems to be a new proof, and it is an open problem to give analogous proofs for the other affine Weyl groups.

\begin{corollary}{\bf (Bott)}\label{c:bottdenom}
We have
    \[ \sum_{w \in \widetilde{S}_n^\circ} q^{\ell(w)} = \frac{1}{(1-q)(1-q^2) \cdots (1-q^{n-1})}. \]
\end{corollary}
\begin{proof}
By the development above, the encoding series for the gap vectors is
\[ \sum_{g \in \mathbb{Z}_{\geq 0}^{n-1}} x_1^{g_1} x_2^{g_2} \cdots x_{n-1}^{g_{n-1}} = \frac{1}{(1-x_1)(1-x_2) \cdots (1-x_{n-1})}. \]
By Proposition~\ref{p:coxeter_length}, we can then obtain the length generating series by substituting $q^{n-i}$ for $x_i$.  This yields the result.
\end{proof}

To uncover the polyhedral structure that will be useful in conjunction with
patterns, we need a further refinement.  We say that an abacus on $n$ columns
is {\bf minimal} if its delta vector uses only entries between $1$ and $n-1$.
For example, the minimal abaci in $n = 4$ are shown below in Figure~\ref{f:minab}.

\begin{figure}[ht]
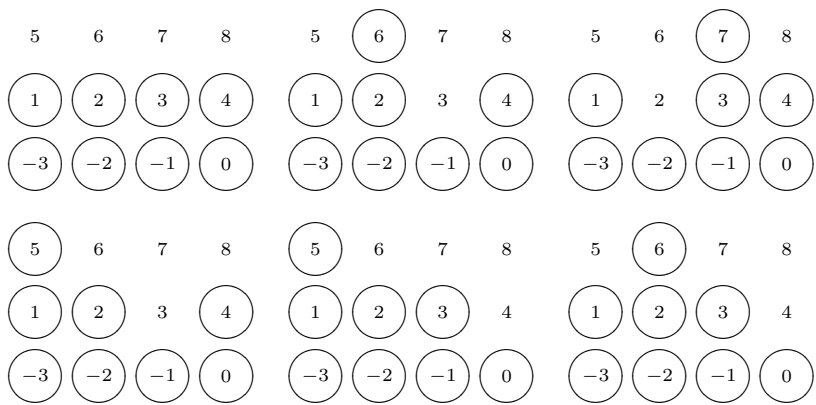

\[
\abacus {   
\nc{5} &  \nc{6} & \nc{7} & \nc{8} \\
\ci{1} &  \ci{2} & \ci{3} & \ci{4} \\
\ci{-3} &  \ci{-2} & \ci{-1} & \ci{0} \\
}
\hspace{0.1in}
\abacus {   
\nc{5} &  \ci{6} & \nc{7} & \nc{8} \\
\ci{1} &  \ci{2} & \nc{3} & \ci{4} \\
\ci{-3} &  \ci{-2} & \ci{-1} & \ci{0} \\
}
\hspace{0.1in}
\abacus {   
\nc{5} &  \nc{6} & \ci{7} & \nc{8} \\
\ci{1} &  \nc{2} & \ci{3} & \ci{4} \\
\ci{-3} &  \ci{-2} & \ci{-1} & \ci{0} \\
}
\]
\[
\abacus {   
\ci{5} &  \nc{6} & \nc{7} & \nc{8} \\
\ci{1} &  \ci{2} & \nc{3} & \ci{4} \\
\ci{-3} &  \ci{-2} & \ci{-1} & \ci{0} \\
}
\hspace{0.1in}
\abacus {   
\ci{5} &  \nc{6} & \nc{7} & \nc{8} \\
\ci{1} &  \ci{2} & \ci{3} & \nc{4} \\
\ci{-3} &  \ci{-2} & \ci{-1} & \ci{0} \\
}
\hspace{0.1in}
\abacus {   
\nc{5} &  \ci{6} & \nc{7} & \nc{8} \\
\ci{1} &  \ci{2} & \ci{3} & \nc{4} \\
\ci{-3} &  \ci{-2} & \ci{-1} & \ci{0} \\
}
\]
\caption{The $(n-1)!$ different minimal abaci in $n=4$}\label{f:minab}
\end{figure}

\begin{proposition}
    There are $(n-1)!$ distinct minimal abaci on $n$ columns.
\end{proposition}
\begin{proof}
    We argue by induction, the result being clear if $n = 2$.  Assume the formula holds for abaci on $n-1$ columns.  To form a minimal abacus on $n$ columns, we can start with a minimal abacus on $n-1$ columns, insert a new column containing a new largest defining bead in any of $n-1$ distinct positions, and rebalance the resulting $n$ column abacus.  Moreover, every minimal abacus on $n$ columns arises this way.  Hence, the formula holds by induction.
\end{proof}

Given $w \in \widetilde{S}_n^\circ$, we can project $w$ to a minimal abacus by
repeatedly removing multiples of $n$ entries between consecutive defining beads
and then rebalancing the diagram.  We call the minimal abacus obtained in this
way the {\bf bias} of $w$.  Equivalently, the bias $b$ of $w$ is
specified by its delta vector $\bar{b} = (\bar{w}_1 \mod n, \bar{w}_2 \mod n,
\ldots, \bar{w}_{n-1} \mod n)$.  Let  $\mathsf{BIAS}_n$ denote the set of
$(n-1)!$ possible biases on $n$ columns.

\begin{example}
    The bias of the abacus shown in Figure~\ref{ex:1} is given by $\bar{b} = (4,4,1,4,4)$.
\end{example}

\begin{lemma}\label{l:thec}
We can decompose the set of gap vectors into a disjoint union of {\it shifted dilated cones} 
\[ \widetilde{S}_n^\circ \cong \mathbb{Z}_{\geq 0}^{n-1} = \dunion_{b \in \mathsf{BIAS}_n} \mathcal{C}_{b}^n \]
where 
\[ \mathcal{C}_{b}^n := \{ ( t_1, 2 t_2, \ldots, (n-1) t_{n-1}) + (\opdot{b}_1, \ldots,  \opdot{b}_{n-1}) : t_i \in \mathbb{Z}_{\geq 0} \}. \]
\end{lemma}

Note that each element of $\mathcal{C}_b^n$ is shifted by the same vector
$(\opdot{b}_1, \ldots,  \opdot{b}_{n-1})$ that depends only on $b$.  Hence we
will refer to points in $\mathcal{C}_b^n$ by their $t$-coordinates, (ab)using
the notation $(t_1, \ldots, t_{n-1}) \in \mathcal{C}_b^n$.

\begin{proof}
We claim that each gap vector $\opdot{w} = (\opdot{w}_1, \ldots, \opdot{w}_{n-1}) \in \mathbb{Z}_{\geq 0}^{n-1}$
exists in precisely one of the $\mathcal{C}_b^n$ sets.
To see this, draw the abacus associated to the point $w$.  Suppose the $i$ and
$(i+1)$st defining beads have more than $n$ entries between them.  Then we delete
one entire level of the array between them (and then renumber the remaining
entries of the array).  This maneuver removes $i$ gaps from the coordinate $\opdot{w}_i$
(since there will be $n-i$ beads on the level we remove), which is equivalent
to removing $1$ from coordinate $t_i$.  Repeat this process until every
consecutive pair of defining beads is separated by less than $n$ entries, and
then rebalance the abacus (by subtracting the sum of the levels of the defining
beads from the position of each defining bead).  By definition, the result will
be one of the minimal abaci.

Moreover this process is reversible since we can recover $w$ by starting with
the minimal abacus, inserting levels as prescribed by the $t_i$ coordinates,
and rebalancing.  Hence, the point of $\mathcal{C}_b^n$ is unique.
\end{proof}

Until now, we have focused on the minimal length coset representatives
$\widetilde{S}_n^\circ$.  From the length-additive parabolic decomposition in
the theory of Coxeter groups, we have that the base-window notation of each $w
\in \widetilde{S}_n$ can be decomposed into a set of values together with a
``sorting permutation'' $v \in S_n$.  The set of values is represented by some
$u \in \widetilde{S}_n^\circ$, and we have seen that these further decompose
into subsets of elements having the same bias.  The finite permutation $v$ is
the unique finite permutation having entries in the same relative order as the
base-window notation for $w$.  We call $v$ the {\bf flattening} of $w$, and it
follows that $\ell(w) = \ell(u) + \ell(v)$.  

Hence, we can extend our polyhedral embedding to $\widetilde{S}_n$ by simply
taking $n!$ copies of the embedding for $\widetilde{S}_n^\circ$.  Thus
we let $\mathcal{C}^n_{b,v}$ denote the set of $w \in \widetilde{S}_n$ whose
bias is $b$ and whose flattening is equal to $v$.  

\begin{corollary}
    We have the disjoint union
    \[ \widetilde{S}_n \cong \dunion_{\substack{b \in \mathsf{BIAS}_n\\ v \in S_n}} \mathcal{C}_{b,v}^n. \]
\end{corollary}

\begin{example}
In $n = 3$, we can draw the minimal length coset representatives as a set of
lattice points in the plane.  There
are two minimal abaci, given by $\bar{a} = (1,1)$ and $\bar{b} = (2,2)$ with
offsets given by $\opdot{a} = (0,0)$ and $\opdot{b} = (0,1)$, respectively.
Then the gap vectors $\widetilde{S}_3^\circ \cong \mathbb{Z}_{\geq 0}^2$ are a
disjoint union of two cones, where the second coordinate has been dilated by
$2$ and the cones have been shifted by $(0,0)$ and $(0,1)$, respectively.

\bigskip
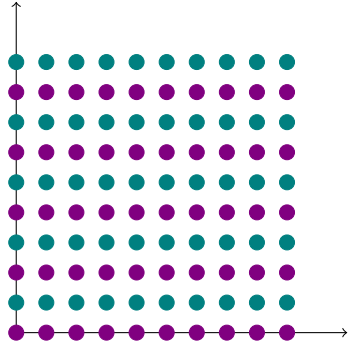
\begin{center}
\begin{tikzpicture}[scale=0.4]
    \draw [->] (0,0) -- (11,0);
    \draw [->] (0,0) -- (0,11);
    \foreach \x in {0,1,...,9}{
      \foreach \y in {0,1,...,4}{
        \node[draw,circle,inner sep=2pt,fill,violet] at (\x,2*\y) {};
        \node[draw,circle,inner sep=2pt,fill,teal] at (\x,2*\y+1) {};
      }
    }
\end{tikzpicture}
\end{center}
\bigskip

The entire affine symmetric group $\widetilde{S}_3$ consists of six copies of this set of lattice points, one for each choice of flattening.
\end{example}

\bigskip
\subsection{Patterns}

Fix a permutation pattern $p \in S_k$, together with a bias $b$ and flattening
$v$.  We will first explain how to characterize the elements of
$\mathcal{C}_{b,v}^n$ that contain an instance of $p$.  Recall that an instance
of the pattern $p$ in the affine permutation $w$ is a choice of $k$ positions
in the $\mathbb{Z}$-notation for $w$ whose values have the same relative order
as $p$.  To coordinatize this, we consider two pieces of data associated to an
instance:  a strand assignment, and a window assignment.

\begin{definition}
Let the {\bf strand assignment} of an instance be the function $\pi$ assigning
each entry of $p$ to an entry of the base window of $w$, where $\pi(i)=j$ means
that $p_i$ is represented by some positional translation of the $j$th largest
value of the base window (where $j = n$ represents the largest value).
\end{definition}

The set of potential strand assignments for $p$ is finite, consisting of all
sequences of length $k$ with values from $\{1, \ldots, n\}$.  In fact, it is
not difficult to discover a further requirement for strand assignments.

\begin{lemma}
    Let $\pi$ be a strand assignment for $p$.  Then either every inversion in $p$ must correspond to a strict inversion in $\pi$, or else the set of $w$ containing $p$ with strand assignment $\pi$ is empty.
\end{lemma}
\begin{proof}
    Suppose $i < j$ and $p_i > p_j$.  If $\pi(i) \leq \pi(j)$ then the elements representing $p_i$ and $p_j$ would necessarily be increasing in $w$, a contradiction.
\end{proof}

\begin{definition}
The {\bf window assignment} of an instance is the vector $(c_1, \ldots,
c_{k-1})$ where $c_i$ is the number of positional window boundaries between the
entries representing $p_i$ and $p_{i+1}$ in $w$.  If $p_i$ and $p_{i+1}$ lie
within the same window then we set $c_i = 0$.  
\end{definition}

\begin{example}\label{e:9411}
Consider the highlighted instance of $p = [24351]$ in the $\mathbb{Z}$-notation for $w = [-9, 4, 11]$:
\[ (\cdots -15, {\bf -2}, {\bf 5} | -12, {\bf 1}, 8 | -9, 4, 11 | -6, {\bf 7}, 14 | {\bf -3}, 10, 17 | \cdots) \]
Then, $\pi = [2,3,2,2,1]$ and $c = (0, 1, 2, 1)$.
\end{example}

It is clear that the strand assignment and the window assignment completely
determine a pattern instance in $w$.  Given an affine permutation $w$ in
$\mathbb{Z}$-notation, we can recover the $t$-coordinates from
$\mathcal{C}_{b}^n$ as follows.

\begin{lemma}\label{l:ztot}
Given $w \in \widetilde{S}_n^\circ$ let $j \geq 0$ be maximal such that $w(i+1) > w(i+jn)$.  Then, $t_i$ equals the number of window boundaries lying between $w(i+1)$ and $w(i+jn)$.
\end{lemma}

For Example~\ref{e:9411}, we find $t_2 = 2$ since there are two window boundaries lying between entries $11$ and $10$.  Similarly, $t_1 = 4$.

\begin{proof}
Work by induction starting from a minimal abacus.  In a minimal abacus, there
are no adjacent inversions between windows in $\mathbb{Z}$-notation and all the
$t_i$ are zero.  Each time we add one to a $t_i$ coordinate, we adjust $w$ by
adding $n$ to each of the $n-i$ largest values in the base window, and then
subtracting $n-i$ from each of the values in the base window (to rebalance).
This places one new window boundary that is counted by the description in the
statement, and preserves all of the other window boundaries.
\end{proof}

We next characterize the $t$-coordinates of points in $\mathcal{C}_{b,v}^n$
that contain an instance of $p$ with strand assignment $\pi$.  To accomplish
this, we highlight some data in $(p, \pi)$.

\begin{definition}
    Given a pattern $p \in S_k$ and a strand assignment $\pi$ for $p$, we say that an {\bf upshift} is a pair $j < i$ such that $p_j = p_i + 1$ and $\pi(j) > \pi(i)$.  A {\bf downshift} is a pair $i < j$ such that $p_j = p_i + 1$ and $\pi(i) > \pi(j)$.
\end{definition}

\begin{example}
    Consider $(p, \pi) = ([24351], [2,3,2,2,1])$.  The values $1$ and $2$ in $p$ form an upshift that we denote (positionally) as $(1<5)$.  Values $2$ and $3$ are both assigned to strand $2$, so no shift takes place.  Values $3$ and $4$ form the upshift $(2<3)$, and values $4$ and $5$ form the downshift $(2<4)$.
\end{example}

We are now in a position to prove our main result in this section.

\begin{theorem}\label{t:polyhedron}
For each $(p, \pi)$, the set
\[ \{t = (t_1, \ldots, t_{n-1}) \in \mathcal{C}_{b,v}^n : \text{ $t$ contains an instance of $p$ using strand assignment $\pi$} \} \]
consists of the integer points in a rational polyhedron.
\end{theorem}
\begin{proof}
We first prove that the set
\[ \{(t, c) \in \mathcal{C}_{b,v}^n \times \mathbb{Z}_{\geq 0}^{k-1} : \text{ $t$ contains an instance of $p$ using strand assignment $\pi$ and window assignment $c$} \} \]
consists of integer points in a rational polyhedron.  Then, we can project onto
the $t$-coordinates to obtain the result.  More precisely, we will give a
collection of integral linear inequalities that the $(t,c)$ coordinates satisfy
exactly when they describe a valid instance of $p$ in the affine permutation
corresponding to $t$.  Then the fundamental Minkowski--Weyl Theorem for convex
polyhedra allows us to describe this set as the Minkowski sum of a (bounded)
polytope and a recession cone of rays.  To perform the projection, we simply
ignore the $c$ coordinates in this latter description.  (See \cite{beck-robins}
or \cite[Lecture 1]{ziegler} for an introduction to these ideas.)

Given a window assignment $c$, we imagine placing the values of $p$ into an
affine permutation $w$ (whose values are to be determined) in increasing order.
Whenever we place a larger value in a position to the right, on the same strand
or higher, we impose no conditions on the $t_i$ because the strands are necessarily
increasing in $w$.  However, if we place a larger value to the left, we must
increase the strand and so this pair of consecutive values is an upshift.
Also, if we place a larger value to the right on a lower strand, this pair of
consecutive values is a downshift.  These do impose conditions on the $t_i$.

By Lemma~\ref{l:ztot}, we have that $t_i$ represents the maximal number of positional window boundaries lying between
(translates of) the $i$th and $(i+1)$st largest elements of the base window that can form an inverted pair.
Similarly, $t_i + t_{i+1} + \cdots t_{j-1} + \lf \bar{b}_{i,j} \rf$ represents the maximal
number of positional window boundaries lying between translates of the $i$th and $j$th
largest elements of the base window that form an inverted pair.  
Here, $\lf \bar{b}_{i,j} \rf = \lf \frac{ \bar{b}(i) + \bar{b}(i+1) + \cdots + \bar{b}(j-1)
}{n} \rf$ represents the (constant) contribution from the bias.

Therefore, when we have an upshift from strand $\pi(i)$ to strand $\pi(j)$, we
must ensure that $t_{\pi(i)} + t_{\pi(i)+1} \cdots + t_{\pi(j)-1} + \lf
\bar{b}_{\pi(i),\pi(j)} \rf$ is large enough to ensure that the entries at
positions $j$ and $i$ are inverted.  These positions are separated by $c_j +
c_{j+1} + \cdots + c_{i-1}$ window boundaries.  

Hence, for each upshift $(j < i)$ we include an inequality of the form
\[ t_{\pi(i)} + t_{\pi(i)+1} + \cdots + t_{\pi(j)-1} + \lf \bar{b}_{\pi(i),\pi(j)} \rf \geq c_{j} + c_{j+1} + \cdots + c_{i-1}. \]
Using similar reasoning for each downshift $(i < j)$, we include an inequality of the form
\[ t_{\pi(j)} + t_{\pi(j)+1} + \cdots + t_{\pi(i)-1} + \lf \bar{b}_{\pi(j),\pi(i)} \rf \leq c_{i} + c_{i+1} + \cdots + c_{j-1} -1. \]

The choice of flattening enters in the initial conditions on $c_i$, which
specify the minimal number of window boundaries between the $i$ and $(i+1)$st
position in the pattern instance.  Since these positions must be increasing, we
include
\[ c_i \geq \begin{cases} 0 & \text{ if $v^{-1}(\pi(i)) < v^{-1}(\pi(i+1))$ } \\ 1 & \text{ otherwise}\end{cases} \]
for each $1 \leq i < k$.  The initial conditions $t_i \geq 0$ should also be included.

Our strategy to place $p$ into the affine permutation will succeed if all of these linear inequalities are satisfied.  If any are not satisfied, then we will have a pair of consecutive values from $p$ whose representatives in the affine permutation do not faithfully represent the pattern.  Hence, the integer points of this polyhedron form precisely the set given in the beginning of the argument.  After projection, we obtain the result.
\end{proof}

\begin{example}
    The linear inequalities obtained for $(p, \pi) = ([24351], [2,3,2,2,1])$ with $\opdot{b} = (0,0)$ and $v = [123]$ are:
    \[ t_1 \geq c_1 + c_2 + c_3 + c_4, \ \ t_2 \geq c_2, \ \  t_2 \leq c_2 + c_3 - 1, \]
    \[ c_1 \geq 0, c_2 \geq 1, c_3 \geq 1, c_4 \geq 1; \ \ t_1 \geq 0, t_2 \geq 0. \]
\end{example}

\begin{definition}
We will refer to the rational polyhedron constructed in the proof of Theorem~\ref{t:polyhedron} by $\mathcal{C}_{b,v}^n(p, \pi)$.
\end{definition}

It turns out that the bias and flattening parameters do not change the polyhedra very much.

\begin{lemma}\label{l:rays}
    Let $b_0$ be the bias given by $\opdot{b} = (0, 0, \ldots, 0)$ and $v_0$ be the identity permutation in $S_{n}$.  Then for any other choice of $b \in \mathsf{BIAS}_n$ and $v \in S_n$, we have that $\mathcal{C}_{b, v}^n(p, \pi)$ has the same set of infinite rays as $\mathcal{C}_{b_0,v_0}^n(p, \pi)$.
\end{lemma}

As a result, we often drop $b$ and $v$ from our notation, and let $\mathcal{C}^n(p, \pi) = \mathcal{C}_{b_0,v_0}^n(p,\pi)$.

\begin{proof}
Write the polyhedron $\mathcal{C}_{b,v}^n(p, \pi)$ as the solution set to a
collection of linear inequalities.  We can use a matrix $A$ and multiplication
by $-1$ to write this in a standard form $A x \leq b$.  It is well-known (and
straightforward to verify) that changing $b$ cannot change any of the infinite
rays in the solution set.  Since changing the bias or flattening parameters only
alters the defining inequalities by a constant, and preserves all of the
coefficients of the $t_i$ and $c_i$, we obtain the result.
\end{proof}

\begin{corollary}\label{c:rational}
For any permutation pattern $p$ and any $n \geq 2$, the generating series
\[ F_{p, n}(x) = \sum_{\substack{w \in \widetilde{S}_n \\ w \text{ avoids } p}} x^{\ell(w)} \]
is rational.  Equivalently, the coefficient sequence is generated by a linear constant-coefficient recurrence.
\end{corollary}
\begin{proof}
Using Brion's formula (see \cite{beck-robins} or \cite{barvinok}) together with
inclusion-exclusion applied to Theorem~\ref{t:polyhedron}, we can obtain a
rational encoding series for the points of each $\bigcup_{\pi}
\mathcal{C}_{b,v}^n(p, \pi)$.  
The subsequent union of these sets over all $b$ and $v$ are disjoint, so we can
simply add the encoding series together.  Then, we specialize the encoding
series by setting $t_i$ to $\left(x^i\right)^{n-i}$ for each $1 \leq i < n$.
The first exponent dilates the lattice to recover the gap coordinates as in
Lemma~\ref{l:thec}, and the second exponent comes from
Proposition~\ref{p:coxeter_length}.  Finally, we subtract the result from Bott's
formula (which itself is rational) to enumerate the $p$-avoiding elements.
\end{proof}

Let us turn to some examples in $n = 3$ where we can draw pictures.

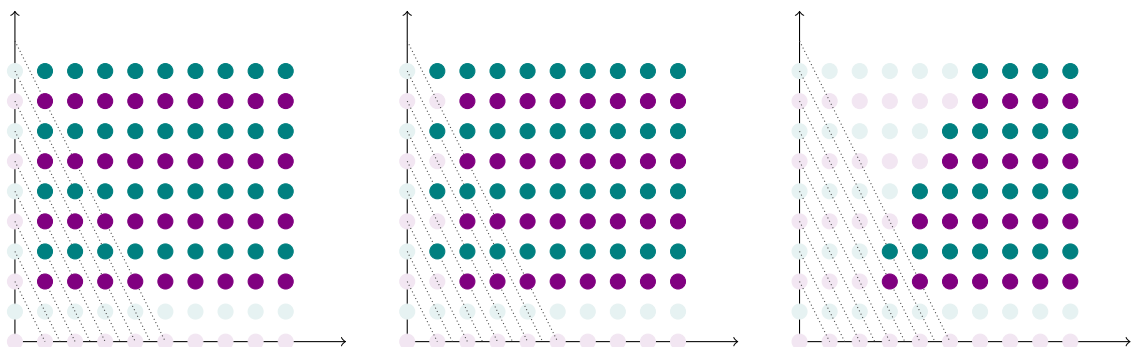
\begin{figure}[ht]
\begin{center}
\begin{tikzpicture}[scale=0.4]
    \draw [->] (0,0) -- (11,0);
    \draw [->] (0,0) -- (0,11);
    \foreach \x in {0,1,...,9}{
      \foreach \y in {0,1,...,4}{
        \node[draw,circle,inner sep=2pt,fill,violet!10] at (\x,2*\y) {};
        \node[draw,circle,inner sep=2pt,fill,teal!10] at (\x,2*\y+1) {};
      }
    }
    \draw [darkgray,densely dotted] (0,2) -- (1,0);
    \draw [darkgray,densely dotted] (0,3) -- (1.5,0);
    \draw [darkgray,densely dotted] (0,4) -- (2,0);
    \draw [darkgray,densely dotted] (0,5) -- (2.5,0);
    \draw [darkgray,densely dotted] (0,6) -- (3,0);
    \draw [darkgray,densely dotted] (0,7) -- (3.5,0);
    \draw [darkgray,densely dotted] (0,8) -- (4,0);
    \draw [darkgray,densely dotted] (0,9) -- (4.5,0);
    \draw [darkgray,densely dotted] (0,10) -- (5,0);
    \foreach \x in {1,2,...,9}{
      \foreach \y in {1,2,...,4}{
        \node[draw,circle,inner sep=2pt,fill,violet] at (\x,2*\y) {};
      }
    }
    \foreach \x in {1,2,...,9}{
      \foreach \y in {1,2,...,4}{
        \node[draw,circle,inner sep=2pt,fill,teal] at (\x,2*\y+1) {};
      }
    }
\end{tikzpicture}
\hspace{0.2in}
\begin{tikzpicture}[scale=0.4]
    \draw [->] (0,0) -- (11,0);
    \draw [->] (0,0) -- (0,11);
    \foreach \x in {0,1,...,9}{
      \foreach \y in {0,1,...,4}{
        \node[draw,circle,inner sep=2pt,fill,violet!10] at (\x,2*\y) {};
        \node[draw,circle,inner sep=2pt,fill,teal!10] at (\x,2*\y+1) {};
      }
    }
    \draw [darkgray,densely dotted] (0,2) -- (1,0);
    \draw [darkgray,densely dotted] (0,3) -- (1.5,0);
    \draw [darkgray,densely dotted] (0,4) -- (2,0);
    \draw [darkgray,densely dotted] (0,5) -- (2.5,0);
    \draw [darkgray,densely dotted] (0,6) -- (3,0);
    \draw [darkgray,densely dotted] (0,7) -- (3.5,0);
    \draw [darkgray,densely dotted] (0,8) -- (4,0);
    \draw [darkgray,densely dotted] (0,9) -- (4.5,0);
    \draw [darkgray,densely dotted] (0,10) -- (5,0);
    \foreach \x in {2,3,...,9}{
      \foreach \y in {1,2,...,4}{
        \node[draw,circle,inner sep=2pt,fill,violet] at (\x,2*\y) {};
      }
    }
    \foreach \x in {1,2,...,9}{
      \foreach \y in {1,2,...,4}{
        \node[draw,circle,inner sep=2pt,fill,teal] at (\x,2*\y+1) {};
      }
    }
\end{tikzpicture}
\hspace{0.2in}
\begin{tikzpicture}[scale=0.4]
    \draw [->] (0,0) -- (11,0);
    \draw [->] (0,0) -- (0,11);
    \foreach \x in {0,1,...,9}{
      \foreach \y in {0,1,...,4}{
        \node[draw,circle,inner sep=2pt,fill,violet!10] at (\x,2*\y) {};
        \node[draw,circle,inner sep=2pt,fill,teal!10] at (\x,2*\y+1) {};
      }
    }
    \draw [darkgray,densely dotted] (0,2) -- (1,0);
    \draw [darkgray,densely dotted] (0,3) -- (1.5,0);
    \draw [darkgray,densely dotted] (0,4) -- (2,0);
    \draw [darkgray,densely dotted] (0,5) -- (2.5,0);
    \draw [darkgray,densely dotted] (0,6) -- (3,0);
    \draw [darkgray,densely dotted] (0,7) -- (3.5,0);
    \draw [darkgray,densely dotted] (0,8) -- (4,0);
    \draw [darkgray,densely dotted] (0,9) -- (4.5,0);
    \draw [darkgray,densely dotted] (0,10) -- (5,0);

    \foreach \x/\y in {3/2, 4/4, 5/6, 6/8, 
                       4/2, 5/4, 6/6, 7/8, 
                       5/2, 6/4, 7/6, 8/8, 
                       6/2, 7/4, 8/6, 9/8,
                       7/2, 8/4, 9/6,
                       8/2, 9/4,
                       9/2 }
    {
        \node[draw,circle,inner sep=2pt,fill,violet] at (\x,\y) {};
    }

    \foreach \x/\y in {3/3, 4/5, 5/7, 6/9, 
                       4/3, 5/5, 6/7, 7/9,
                       5/3, 6/5, 7/7, 8/9,
                       6/3, 7/5, 8/7, 9/9,
                       7/3, 8/5, 9/7,
                       8/3, 9/5,
                       9/3 }
    {
        \node[draw,circle,inner sep=2pt,fill,teal] at (\x,\y) {};
    }
\end{tikzpicture}
\end{center}
\caption{(a) $(p,\pi) = ([321],[3,2,1])$; (b) $(p, \pi) = ([2431], [3,3,2,1])$;\ \hspace{0.7in}\ (c) $(p, \pi) = ([24351], [2,3,2,2,1])$.}\label{f:n3ex}
\end{figure}

\begin{example}
    In Figure~\ref{f:n3ex} we have displayed some $\mathcal{C}_{b,v}^n(p, \pi)$.  In each of the examples, we have $n = 3$, $v = [123]$, both biases are displayed superimposed, and $(p, \pi)$ vary.
    We have also drawn some of the hyperplanes of constant Coxeter length from which the contributions to the rational generating series can be computed.

Observe that in Example (a) the counting sequence for the number of $p$-avoiding
elements eventually stabilizes.  In Example (b), we have a periodic sequence
with period $2$.  Example (c) produces an unbounded counting sequence (although
other strand assignments provide a ray in the $y$ direction that is missing for
this assignment; the full counting sequence for this $p$ turns out to be periodic).
\end{example}

\begin{warning}\label{w:e}
These polyhedra can be empty.  For example, $p = [7,1,0,4,5,2,8,10,6,9,3]$ has only one strand assignment using $3$ strands, and the corresponding $\mathcal{C}^n(p, \pi)$ polyhedron is empty.
\end{warning}

There are some natural questions about these polyhedra to which we do not currently know the answer.

\begin{question}
If we fix the bias and flattening parameters, is the union
    \[ \bigcup_{\pi} \mathcal{C}_{b,v}^n(p, \pi) \]
over all strand assignments necessarily convex? 
(If so, this would dramatically simplify the computation of the encoding series.)
\end{question}

\begin{question}
Given a pattern $p \in S_k$ with $j$ strands, we certainly need $n \geq j$ in
order to successfully embed $p$ into $\widetilde{S}_n$.  By Warning~\ref{w:e},
this inequality is sometimes strict.  Is there a simple way to describe the
minimal size of an affine permutation that contains a given pattern $p$?
\end{question}

\bigskip
\section{Periodic patterns}\label{s:periodic}

Let $a_i$ denote the coefficients of the rational generating series
$F_{p,n}(x)$ from Corollary~\ref{c:rational}.  That is, $a_i$ counts the number
of affine permutations of fixed size $n$ and length $i$ that avoid the fixed
pattern $p$.  Since the $a_i$ obey a recurrence, it follows that there are three
possible types of behavior.

\begin{definition}\label{d:classify}
We say that a permutation pattern $p$ is {\bf finitely enumerated} if the $a_i$
are eventually zero.  We say $p$ is {\bf periodic} if the $a_i$ eventually
satisfy $a_i = a_{i-N}$ for some fixed $N$.  Otherwise, we say that $p$ is {\bf
unbounded}.
\end{definition}

(To verify that this definition is etymologically sound, use the pigeonhole principle to
show that whenever $a_i$ is a bounded sequence that satisfies a recurrence using
a fixed number of prior terms, then $a_i$ is actually periodic.)

Crites' characterized the finitely enumerated patterns in \cite{crites}, and
Hanusa--Jones gave the first example of a periodic pattern, $p = [321]$, in
\cite{HJ}.  Our goal in this section is to characterize all of the periodic patterns.  

Note that the classification in Definition~\ref{d:classify} depends only on the
denominator of the generating series and so the contributions from each bias
and flattening must each fall into the same case by Lemma~\ref{l:rays}.  For
this reason, it suffices to work with the enumerating sequence for
$\bigcup_{\pi} \mathcal{C}^n(p, \pi)$ in this section.

\begin{definition}
Given $p \in S_k$, let $m$ be the length of the longest decreasing subsequence of $p$.  In this situation, we say that $p$ {\bf has $m$ strands}.
\end{definition}

We rephrase Crites' Theorem from \cite{crites} as follows.

\begin{theorem}{\bf (Crites)}  
    In each $n$, the permutation pattern $p$ is finitely enumerated if and only if $p$ has fewer than $3$ strands.
\end{theorem}

The following result then shows that periodic patterns can only exist on three strands.

\begin{proposition}
    In each $n$, if $p$ has four or more strands then $p$ is unbounded.
\end{proposition}
\begin{proof}
Consider Bott's formula for $\widetilde{S}_3^\circ$.  The sequence of
coefficients is unbounded, and the affine permutations in
$\widetilde{S}_3^\circ$ all avoid $p$ (since the length of the longest
decreasing subsequence in any of them is clearly $3$ or less).  When $n > 3$,
we can embed $w \in \widetilde{S}_3^\circ$ into $\widetilde{S}_n^\circ$ by
padding $\opdot{w}$ with zeros on the left.  This embedding is injective, the
length of the longest decreasing subsequence in the image will be the same or
smaller, and by Proposition~\ref{p:coxeter_length} we do not change the Coxeter
length.  Hence, we obtain the result.
\end{proof}

\begin{lemma}\label{l:periodic_classification}
In each $n$, we have that $p$ is periodic if and only if there exists a constant $B$ such that $\bigcup_{\pi} \mathcal{C}^n(p, \pi)$ contains every point $(t_1, \ldots, t_{n-1})$ that has two or more $t_i$ coordinates larger than $B$.
\end{lemma}
\begin{proof}
Since the enumerating sequences are generated by a recurrence, we have that $p$
is periodic if and only if there exists an upper bound for the values of the
sequence.  Also, the condition in the statement for the $t$ coordinates is true
if and only if it is true for the corresponding $\opdot{w}$ gap vector coordinates.

To prove the result, first suppose that at most one $\opdot{w}_i$ coordinate
(from the space $\mathbb{Z}^{n-1}_{\geq 0}$ of gap vectors) can become
arbitrarily large when we avoid $p$.  Then when we intersect with the
hyperplane 
\[ \opdot{w}_{n-1} + 2\opdot{w}_{n-2} + 3\opdot{w}_{n-3} + \cdots +
(n-1)\opdot{w}_{1} = i \]
of points with fixed Coxeter length $i$, for large $i$, the unbounded
coordinate actually becomes determined; it must take up the ``slack'' in this
equation for all of the bounded coordinates.  As a result, we only have a
bounded number of $p$-avoiding gap vectors, so the sequence is periodic.

Conversely, if there can be two unbounded gap vector coordinates when we avoid
$p$, then one of these coordinates will be undetermined when we intersect with
the hyperplane of fixed Coxeter length $i$, for sufficiently large $i$.  Hence,
the enumerating sequence is unbounded and so $p$ is not periodic.
\end{proof}

\begin{corollary}\label{c:3rays}
In $n = 3$, we have that $p$ is periodic if and only if $\bigcup_{\pi} \mathcal{C}^3(p, \pi)$ contains infinite rays in the $t_1$ and $t_2$ directions.
\end{corollary}

We now turn to classify the periodic patterns in $n = 3$.  Eventually we show
that these are the only periodic patterns for any $n$.

\begin{definition}
We say that $(p, \pi)$ is {\bf feasible} if $\bigcup_{\pi} \mathcal{C}^3(p, \pi)$ is nonempty.
\end{definition}

While it remains an open problem to provide a (simple) combinatorial characterization for feasibility, there are standard techniques from convex geometry (such as Fourier--Motzkin elimination or Lenstra's algorithm for integer programming \cite{schrijver}) that may be used to address this question.

\begin{definition}
Let $p \in S_k$ with 3 strands, and let $\pi$ be a strand assignment for $p$.
Consider the diagram of $(p, \pi)$ in which we represent $p_i$ by a point
$(i, p_i)$ in the plane and label the point by its strand assignment
$\pi(i)$.

We say that two elements of the second strand $p_i$ and $p_j$ are {\bf linked
below (above)} if there exists an element of the first (third, respectively)
strand lying below and right (above and left, respectively) of both of them.

We say that two elements of the second strand are {\bf chained below (above)} if
there is a consecutive sequence of elements from the second strand between them that
are linked below (above, respectively).

A {\bf corner} of $(p,\pi)$ consists of a triple $(i,j,k)$ such that $p_i$ and $p_j$ are distinct elements of the second strand, and $p_k$ is an element of the first or third strand that lies inside the square having $p_i$ and $p_j$ as diagonal vertices.

The corner is said to be {\bf tight} if the elements $p_i$ and $p_j$ are chained
below, or chained above.
\end{definition}

Some tight corners that are chained below are shown schematically in Figures~\ref{f:min_schem} and \ref{f:generalform}.  Points drawn in the same row or column can be resolved to a permutation by either perturbation of the points.  Thus, each picture encodes several classical permutation patterns.

\begin{figure}[ht]
\begin{center}
\begin{tikzpicture}[scale=0.7]
    \draw (0,2) -- (2,2);
    \draw[densely dotted] (2,2) -- (2,6);
    \draw (2,2) -- (2,0);
    \draw[dashed] (2,2) -- (8,2);
    \node at (2,2) [circle,draw,fill=white] {$2$};
    \draw (0,5) -- (3,5);
    \draw (3,5) -- (3,8);
    \draw (3,5) -- (6,5);
    \draw (3,5) -- (3,2);
    \node at (3,5) [circle,draw,fill=white] {$3$};
    \draw[densely dotted] (2,6) -- (6,6);
    \draw (6,6) -- (6,8);
    \draw[dashed] (6,6) -- (6,0);
    \draw (6,6) -- (8,6);
    \node at (6,6) [circle,draw,fill=white] {$2$};
    \node at (4,7) [circle,draw,fill=lightgray!40] {$3$};
    \node at (1,4) [circle,draw,fill=lightgray!40] {$3$};
    \node at (4,4) [circle,draw,fill=lightgray!40] {$2$};
    \draw[dashed] (4,4) -- (8,4);
    \draw[dashed] (4,4) -- (4,0);
    \node at (7,1) [circle,draw,fill=lightgray!90] {$1$};
\end{tikzpicture}
\hspace{0.2in}
\begin{tikzpicture}[scale=0.7]
    \draw (0,2) -- (2,2);
    \draw[densely dotted] (2,2) -- (2,6);
    \draw (2,2) -- (2,0);
    \draw[dashed] (2,2) -- (8,2);
    \node at (2,2) [circle,draw,fill=white] {$2$};
    \draw (0,5) -- (3,5);
    \draw (3,5) -- (3,8);
    \draw (3,5) -- (6,5);
    \draw (3,5) -- (3,2);
    \node at (3,5) [circle,draw,fill=white] {$3$};
    \draw[densely dotted] (2,6) -- (6,6);
    \draw (6,6) -- (6,8);
    \draw[dashed] (6,6) -- (6,0);
    \draw (6,6) -- (8,6);
    \node at (6,6) [circle,draw,fill=white] {$2$};
    \node at (4,7) [circle,draw,fill=lightgray!40] {$3$};
    \node at (1,4) [circle,draw,fill=lightgray!40] {$3$};
    \node at (4,4) [circle,draw,fill=white] {$2$};
    \draw[dashed] (4,4) -- (8,4);
    \draw[dashed] (4,4) -- (4,0);
    \node at (5,1) [circle,draw,fill=lightgray!90] {$1$};
    \node at (7,3) [circle,draw,fill=lightgray!90] {$1$};
\end{tikzpicture}
\end{center}
\caption{Minimal tight corners}\label{f:min_schem}
\end{figure}
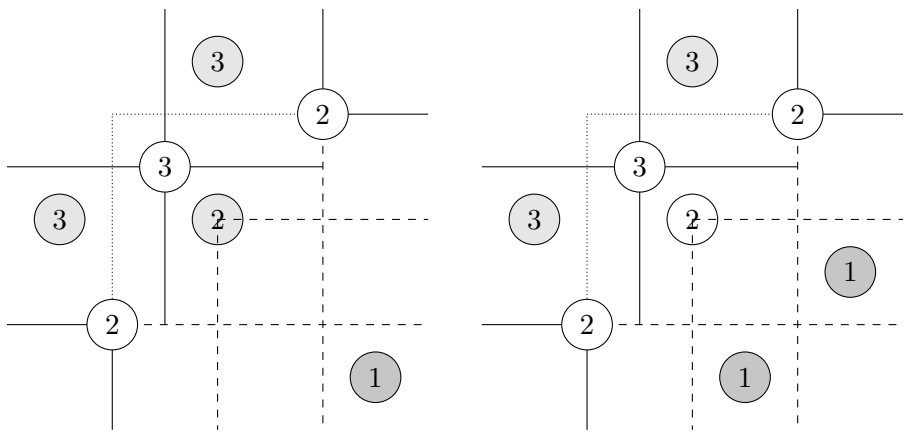

\begin{figure}[ht]
\begin{center}
\begin{tikzpicture}[scale=1.6]
    \draw (1.5,2) -- (2,2);
    \draw[densely dotted] (2,2) -- (2,6);
    \draw (2,2) -- (2,1.0);
    \draw[dashed] (2,2) -- (7,2);
    \node at (2,2) [circle,draw,fill=white] {$2$};
    \draw (1.5,5) -- (3,5);
    \draw (3,5) -- (3,7);
    \draw (3,5) -- (6,5);
    \draw (3,5) -- (3,2);
    \node at (3,5) [circle,draw,fill=white] {$3$};
    \draw[densely dotted] (2,6) -- (6,6);
    \draw (6,6) -- (6,7);
    \draw[dashed] (6,6) -- (6,1.0);
    \draw (6,6) -- (7,6);
    \node at (6,6) [circle,draw,fill=white] {$2$};
    \node at (4,6.5) [circle,draw,fill=lightgray!40] {$3$};
    \node at (1.5,4) [circle,draw,fill=lightgray!40] {$3$};
    \node at (5.5,4.5) [circle,draw,fill=white] {$2$};
    \node at (4.1,3.1) [circle,draw,fill=white] {$2$};
    \node at (3.5,2.5) [circle,draw,fill=white] {$2$};
    \node at (4.9,3.9) {$\vdots$};
    \draw[dashed] (5.5,4.5) -- (7,4.5);
    \draw[dashed] (5.5,4.5) -- (5.5,1.0);
    \draw[dashed] (4.1,3.1) -- (7,3.1);
    \draw[dashed] (4.1,3.1) -- (4.1,1.0);
    \draw[dashed] (3.5,2.5) -- (7,2.5);
    \draw[dashed] (3.5,2.5) -- (3.5,1.0);
    \node at (6.5,4) [circle,draw,fill=lightgray!90] {$1$};
    \node at (5.75,2.8) [circle,draw,fill=lightgray!90] {$1$};
    \node at (4.75,2.25) [circle,draw,fill=lightgray!90] {$1$};
    \node at (3.8,1.5) [circle,draw,fill=lightgray!90] {$1$};
\end{tikzpicture}
\end{center}
\caption{General form:  tight corners}\label{f:generalform}
\end{figure}
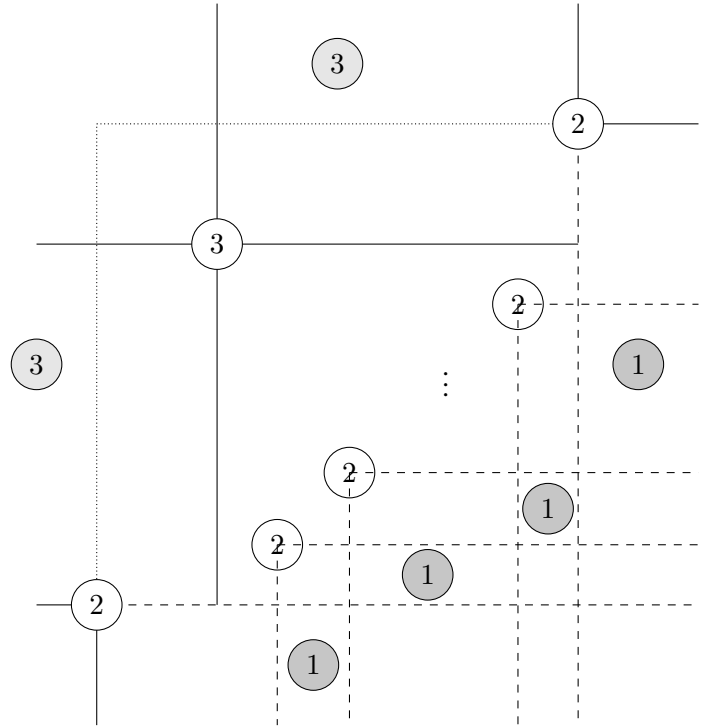

\bigskip

\begin{lemma}\label{l:cinr1}
If $(p, \pi)$ is feasible and contains a tight corner then $t_1$ or $t_2$ is not
a ray of $\mathcal{C}^3(p, \pi)$.
\end{lemma}
\begin{proof}
Suppose $(p, \pi)$ has a tight corner.  Without loss of generality, we may
assume it is chained below as shown in the figures.  Then, we claim that $t_2$
is not a ray.  If it were, we could fix $t_1$ and increase $t_2$ arbitrarily.
However, once $t_1$ is fixed, there is a maximum width for the strand $2$
entries that are chained.  Then we cannot increase $t_2$ past the distance
limited by the strand $3$ entry that is in the tight corner.
\end{proof}

\begin{lemma}\label{l:fir2}
If $(p, \pi)$ is feasible and does not contain a tight corner then both $t_1$
and $t_2$ are rays of $\mathcal{C}^3(p, \pi)$.
\end{lemma}
\begin{proof}
We argue the contrapositive.  Suppose, without loss of generality, that whenever $t_1$ is fixed there are only finitely many values for $t_2$.  If there were no corner of $2$ entries enclosing a strand $3$ entry, then we could separate strands $2$ and $3$, increasing $t_2$ arbitrarily.  If the $2$ entries defining the corner were not chained, then we could slide them along their strand and thereby increase $t_2$ arbitrarily.  Therefore, we must have a tight corner.  
\end{proof}

\begin{theorem}\label{t:ptc}
The pattern $p$ is periodic in $\widetilde{S}_3^\circ$ if and only if there exists a strand assignment $\pi$ that is feasible and does not contain a tight corner.
\end{theorem}
\begin{proof}
First suppose there exists such a strand assignment.  Then, Lemma~\ref{l:fir2}
and Corollary~\ref{c:3rays} imply the result.
    
Next, suppose that no such strand assignment exists.  If this is because no
$\pi$ is feasible, then $p$ is not periodic since the enumeration is given by
Bott's formula.  

So suppose that every feasible $\pi$ has a tight corner.  We show that they all
contain the same type of tight corner (i.e. are all chained above, or chained
below).  Fix some feasible $\pi$ and consider the ``supporting entries'' shown
in light gray in the figures.  If these entries were not present in $p$, it
would be possible to modify the strand assignment $\pi$ to get rid of the tight
corner, a contradiction.

Hence, the supporting entries must be present in $p$.  But this implies that
the strand assignments for the entries of the tight corner are forced in {\em
every} strand assignment.  Therefore, no $\pi'$ can contain the ray that is
missing due to the tight corner of $\pi$ and Lemma~\ref{l:cinr1}.  
Thus, $p$ is not periodic by Corollary~\ref{c:3rays}.
\end{proof}

Finally, we complete the periodic classification for $n > 3$.

\begin{theorem}
Let $n > 3$.  If $p$ is not periodic in $\widetilde{S}_3^\circ$ then $p$ is not periodic in $\widetilde{S}_n^\circ$.
\end{theorem}
\begin{proof}
Suppose for the sake of contradiction that $p$ is not periodic in
$\widetilde{S}_3^\circ$, but $p$ is periodic in $\widetilde{S}_n^\circ$.  Then,
the values of the enumerating sequence for the $p$-avoiding elements of
$\widetilde{S}_n^\circ$ are bounded by some value $B$.

We embed points $(t_1, t_2)$ from $\widetilde{S}_3^\circ$ into
$\widetilde{S}_n^\circ$ by appending zeros on the left.  
After dilation, the length formula for these points is
\[ \ell(0, 0, \ldots, 0, t_1, t_2) = (n-1) t_2 + 2(n-2) t_1. \]
Hence, the points having fixed length $k$ in $\widetilde{S}_n^\circ$ satisfy
\[ t_2 = -\frac{2(n-2)}{n-1} t_1 + \frac{1}{n-1}k. \]

Since $p$ is not periodic in $\widetilde{S}_3^\circ$, we have that
$\bigcup_{\pi} \mathcal{C}^3(p, \pi)$ does not contain rays in both the $t_1$
and $t_2$ directions.  Since the points of $\bigcup_{\pi} \mathcal{C}^3(p, \pi)$
must lie in the nonnegative quadrant, this implies that some ray with positive
slope separates the $p$-containing lattice points from the $p$-avoiding lattice
points in $\mathbb{Z}_{\geq 0}^2$.

Because the embedded points having fixed length $k$ in $\widetilde{S}_n^\circ$
have negative rational slope, there must eventually be some large
value of $k$ for which we obtain more than $B$ distinct points $(t_1, t_2)$
such that:
\begin{enumerate}
    \item[(1)]  $(0, \ldots, 0, t_1, t_2)$ has length $k$ in
        $\widetilde{S}_n^\circ$.
    \item[(2)]  $(t_1, t_2)$ avoids $p$ in $\widetilde{S}_3^\circ$.
\end{enumerate}

Since there are more than $B$ of these embedded points, some of them must
contain $p$.  So suppose $t = (0, \ldots, 0, t_1, t_2)$ contains $p$.  Then
there exists some strand assignment $\pi$ for which $t$ is feasible in
$\mathcal{C}^n(p, \pi)$.  Note that since the first $n-3$ coordinates are zero,
we cannot have any upshifts involving strands $\{1, 2, \ldots, n-2\}$.
Therefore, we can form a strand assignment $\pi'$ by changing all of these
values to $n-2$, and the point $t$ will still be feasible for $\mathcal{C}^n(p,
\pi')$.  But then the point $t$ would also have been feasible for
$\mathcal{C}^3(p, \pi'')$, where $\pi''$ is obtained from $\pi'$ by sending
$n-2$ to $1$, $n-1$ to $2$ and $n$ to $3$.  This contradicts (2) above.
\end{proof}

\begin{theorem}
Let $n > 3$.  If $p$ is periodic in $\widetilde{S}_3^\circ$ then $p$ is periodic in
$\widetilde{S}_n^\circ$.
\end{theorem}
\begin{proof}
By Theorem~\ref{t:ptc}, there exists a feasible strand assignment $\pi$ whose 
solutions include both $t_1$ and $t_2$ as rays, so the solutions include all
points having both coordinates larger than some $B'$.  Let $t' = (t_1, \ldots,
t_{n-1}) \in \bigcup_{\pi} \mathcal{C}^n(p, \pi)$ with two coordinates
larger than $B'$, say $t_i$ and $t_j$.  

Given $\pi$, we can form $\pi'$ from $\pi$ by preserving the images equal to $1$,
replacing the images equal to $2$ by $j$, and replacing the images
equal to $3$ by $n$.
This has the effect of replacing every instance of $t_1$ by $t_1 + \cdots + t_{j-1}$ and every instance of $t_2$ by $t_j + \cdots + t_{n-1}$ in the upshift and downshift inequalities from the proof of Theorem~\ref{t:polyhedron}.

Hence, because the solutions for $\pi$ include all points with $t_1 > B'$ and $t_2 > B'$, the solutions for $\pi'$ include all points
\[ t_1 + \cdots + t_i + \cdots + t_{j-1} > B' \]
and
\[ t_{j} + \cdots + t_{n-1} > B'. \]
Therefore, $t'$ is a feasible point for $\pi'$.

Thus, $\widetilde{S}_n^\circ$ satisfies Lemma~\ref{l:periodic_classification}, so is
periodic.
\end{proof}

\bigskip
\section*{Acknowledgements}

We are grateful to Chris Hanusa, Edwin O'Shea, Len Van Wyk, and Alex Woo for helpful
conversations about this work.  Some problems related to the topic of this
paper were investigated in undergraduate research programs mentored by the
author at James Madison University, and we thank students Traymon Beavers,
Karina Bekova, Abigail Liskey, and Ryan Stees for their contributions.



\def\cprime{$'$}

\end{document}